\documentclass[letterpaper, 10 pt, conference]{ieeeconf}  % Comment this line out if you need a4paper

\IEEEoverridecommandlockouts                              % This command is only needed if
\usepackage{amsfonts,bm}
\usepackage{amssymb,amsmath}
\usepackage{graphicx}
\usepackage{cite}
\usepackage{algorithm2e}
\usepackage{algpseudocode}
\RestyleAlgo{ruled}
\usepackage{ifthen}     

\renewcommand{\epsilon}{\varepsilon}
\renewcommand{\phi}{\varphi}
\renewcommand{\d}{\,\mathrm{d}}
\newtheorem{proposition}{Proposition}

\newtheorem{definition}{Definition}

\renewcommand{\epsilon}{\varepsilon}
\renewcommand{\phi}{\varphi}
\renewcommand{\d}{\,d}

\newcommand{\R}{\mathbb{R}}

\title{\LARGE \bf
Optimal control of diffusion processes:\\ $\infty$-order variational analysis and numerical solution$^*$
}

\author{Roman Chertovskih, Nikolay Pogodaev, Maxim Staritsyn, and A. Pedro Aguiar, \IEEEmembership{Member, IEEE}% <-this % stops a space
\thanks{The authors acknowledge the financial support of the Foundation for Science and
Technology (FCT, Portugal) in the framework of the Associated Laboratory ARISE (LA/P/0112/2020), R\&D Unit SYSTEC (base UIDB/00147/2020 and programmatic UIDP/00147/2020 funds), and projects MLDLCOV (DSAIPA/CS/0086/2020), RELIABLE (PTDC/EEI-AUT/3522/2020). A part of the computations was carried out on the OBLIVION Supercomputer (Évora University) under FCT computational project 2022.15706.CPCA.
}
\thanks{Roman Chertovskih, Maxim Staritsyn, and A. Pedro Aguiar are with Research Center for Systems and Technologies (SYSTEC), ARISE, Department of Electrical
and Computer Engineering,
Faculdade de Engenharia, Universidade do Porto, Rua Dr. Roberto Frias, s/n 4200-465, Porto, Portugal (e-mail: roman@fe.up.pt, staritsyn@fe.up.pt, pedro.aguiar@fe.up.pt). Nikolay Pogodaev is with Dipartimento di Matematica ``Tullio Levi-Civita'' (DM), University of Padova, Via Trieste, 63 - 35121 Padova, Italy (e-mail: nickpogo@gmail.com)
}%
}

\begin{document}
\maketitle
\thispagestyle{empty}
\pagestyle{empty}
%%%%%%%%%%%%%%%%%%%%%%%%%%%%%%%%%%%%%%%%%%%%%%%%%%%%%%%%%%%%%%%%%%%%%%%%%%%%%%%%
\begin{abstract}
We tackle a nonlinear optimal control problem for a stochastic differential equation in Euclidean space and its state-linear counterpart for the Fokker-Planck-Kolmogorov equation in the space of probabilities. Our approach is founded on a novel concept of local optimality surpassing Pontryagin's minimum, originally crafted for deterministic optimal ensemble control problems. A key practical outcome is a rapidly converging numerical algorithm, which proves its feasibility for problems involving Markovian and open-loop strategies.

\end{abstract}
%%%%%%%%%%%%%%%%%%%%%%%%%%%%%%%%%%%%%%%%%%%%%%%%%%%%%%%%%%%%%%%%%%%%%%%%%%%%%%%%
\section{INTRODUCTION}\label{sec:I}

We explore an optimal control problem 
\begin{align*}
(SP) \quad \mathcal \min\big\{\mathcal I[\bm u] 
%\doteq \mathcal I\big[\bm u, X[\bm u]\big] 
\colon \bm u \in \bm U\big\},
\end{align*}
where $\mathcal I[\bm u]$ is defined based on the solution $X=X[\bm u]$ of a 
nonlinear \^{I}to stochastic differential equation (SDE) % given in the integral form:
\begin{align}
X_t = X_0 + \int_0^t f_s[\bm u](X_s)\d s + \int_0^t\sigma_s[\bm u](X_s) \d W_s\label{SDE}
\end{align}
with controlled deterministic coefficients
%\emph{drift} 
$f\colon I \times \R^n \times \bm U  \to \R^n$ and 
%\emph{diffusion} 
$\sigma\colon I \times \R^n \times \bm U \to \R^{(k\times n)}$. Here, $W \colon %(t, \omega) \to W_t(\omega)$, $
I \times \Omega \to \R^k$, is a standard \emph{Wiener process} with independent components $W^j$, defined on a complete probability space (p.s.) $(\Omega, \mathcal F, \mathbb P)$ with a natural filtration $t \mapsto \mathcal F_t^W$ of the sigma-algebra $\mathcal F$, and $X_0\colon %\omega \to X_0(\omega)$, $
\Omega \to \R^n$, is a %given %$\mathcal F$-measurable 
random variable, independent of $\mathcal F_T$. 

The \emph{state trajectories} $X$ %of \eqref{SDE} 
are random processes %$(t, \omega) \mapsto X_t(\omega)$, 
$I \times \Omega \to \R^n$ modulated by \emph{control functions} $\bm u$.
%$\bm u\colon t \mapsto \bm u_t$, $I \to \bm U$, valued in a given (possibly, functional\footnote{For this reason, we outline the dependence on the control parameter $\upsilon \in \bm U$ by writing $f[\bm u](x)$ instead of $f(x, \upsilon)$.}) set $\bm U$. 
We assume that $X$ are progressively measurable with respect to (w.r.t.) the sigma-algebras $\mathcal F_t^{W,X_0}$ generated by random variables $W_t$ and $X_0$. % (adapted to $\mathcal F_t^{W, X_0}$). 

%The coefficients depending on a (possibly, functional\footnote{For this reason, we outline the dependence on the control parameter $\upsilon \in \bm U$ by writing $f[\bm u](x)$ instead of $f(x, \upsilon)$.}) external parameter within a given set $\bm U$. 

 Our objective functional reads
\[
\mathcal I[\bm u]= \mathbb E \, \left[\ell (X_T[\bm u])+ \int_{I} R_s[\bm u](X_s[\bm u]) \d s\right],
\]
where $\ell \colon \R^n \to \R$ and $R\colon I \times \R^n \times \bm U \to \R$ are given %functions called the 
\emph{terminal} and \emph{running costs}, and $\mathbb E$ is the $\mathbb P$-expectation. %, respectively, and $\mathbb E \doteq \mathbb E^{\mathbb P}$ denotes the expectation w.r.t. the canonical measure $\mathbb P$. 

A class $\bm U$ of control inputs, which is pertinent to the problem $(SP)$, consists of random processes $\bm u\colon % (t, \omega) \mapsto \bm u_t(\omega)$, $
I \times \Omega \to U$, defined on the same p.s., adapted to $\mathcal F^{W, X_0}$, and valued in a certain set $U\subseteq \R^m$. However, the  practice of control engineering usually leans towards more ``realizable'' options, whose choice is shaped by the nature and quality of online information available to the decision-maker. We investigate two  ``extreme'' scenarios: 

1) The state $X_t$ is fully observable throughout the entire interval $I$, and the guide has technical options to adjust the control strategy accordingly. In this case, $\bm u$ takes the form of so-called \emph{Markovian strategy} $\bm u_t(\omega) = w_t(X_t(\omega))$ given by a %Borel measurable 
map $w\colon I \times \R^n \to U$. 

2) No observation/intervention is possible during the actual execution of the control process. The guide is expected to pre-define an \emph{open-loop} control strategy $\bm u_t(\omega) \equiv u(t)$, $u \colon I \to U$, relying on the knowledge of the model and insights gained from preliminary experiments. % for some $u \in L_\infty(I;U)$. 

Section~\ref{sec:Markov} focuses on the first option, which is canonical for the stochastic context. The second choice, although somewhat unconventional, finds motivation in different applications where one searches for a ``robust'' or ``broadcast'' control signal that compensates for noise or uncertainty. Notably, the related control and optimization problems, termed (optimal) \emph{ensemble control}, prove to be significantly more intricate compared to those involving Markovian strategies. Section~\ref{sec:robust} delves into certain aspects of this intricacy.

\subsection{Contribution and Novelty}

In the realm of optimal control for diffusion processes, three mainstream frameworks basically stand out: Stochastic Pontryagin's Maximum Principle (SPMP) \cite{BENSOUSSAN1983387,Peng,Hu2021,Pham2021}, Distributed Pontryagin's Principle \cite{WeinanJiequn,Siska,Annunziato2013,Roy2016,Roy2018,Breitenbach2020,Fleig2016,Fleig2017,Borzi2011,Anita2021} assuming a shift to an optimal control of a Fokker-Planck-Kolmogorov equation (FP-PMP), and classical Dynamic Programming (DP) \cite{krylov2008controlled,Annunziato2014,Pham2021,yong1999stochastic} (the bibliography is vast, we only mention a few). 

All these approaches pose significant challenges in numerical contexts, addressed, e.g., by \cite{Roy2018,Roy2016,Breitenbach2020,Annunziato2013, Annunziato2014,Borzi2011,Sinigaglia2021OptimalCO}: 
SPMP deals with a sophisticated system of coupled forward-backward SDEs (FBSDEs) \cite{ma1999forward}, prompting the search for indirect algorithms. The latter two demand a numerical solution to (respectively, linear Fokker-Planck and nonlinear Hamilton-Jacobi) partial differential equations (PDEs), a persistent issue practically viable only for small dimensions. 

In general, Pontryagin's principle (PMP) gives rise to standard indirect algorithms reminiscent of the conventional gradient descent, which incorporate internal ``step'' parameters subjected to line search through specific backtracking, see, e.g. \cite{Breitenbach2020,AMO}. Typically, such algorithms take numerous %(hundreds) 
iterations to achieve a satisfactory approximation of a \emph{local} solution. DP yields a \emph{global} solution but, in practice, it is applicable exclusively to deterministic initial states: when $X_0$ is uncertain, the corresponding Hamilton-Jacobi equation is formulated on the space of measures, even if $\sigma \equiv 0$ \cite{MarigQuin2018}.

In this work, we promote an alternative approach that falls somewhere between FP-PMP and DP. In line with the former, a numerical method derived from our approach involves solving a linear parabolic PDE or extracting statistics through Monte Carlo as pivotal steps in each iteration. Nevertheless, the method is devoid of free intrinsic parameters and shows noticeably fast convergence. This leads to a substantial reduction in computation cost for a local solution, yet understood in a \emph{stronger} sense than that inherent to PMP (see Sec.~\ref{sec:geometric}).

 Recently, our approach showcased its efficacy in numerical optimization involving random ordinary differential equations (ODEs) % and linear continuity equations in the space of probabilities 
 \cite{SChPP-2022, SPP-2023, CPSA2023}. The ongoing extension to the stochastic framework, further complicated by the ``arrow of time'' burden, remained a natural and intriguing challenge. {To our knowledge, this work is pioneering on this line}.

\subsection{Notations, Standing Assumptions, and Preliminaries}\label{sec-hypo}

We use the following \underline{notations}: $\R^n$ is the $n$-dimensional Euclidean space with a fixed norm $|\cdot|$, and the corresponding matrix norm is denoted by the same symbol; $\mathcal L^n$ stands for the Lebesgue measure on $\mathbb R^n$. 

$C(\mathcal X; \mathcal Y)$ denotes the space of continuous functions between normed spaces $\mathcal X$ and $\mathcal Y$, $C(\mathcal X)\doteq C(\mathcal X; \R)$, and $\|\cdot\|_\infty$ is the natural $\sup$-norm on $C$. 

$C^k$ is the space of $k$ times continuously differentiable functions, and $C^\infty_c$ the class of smooth and compactly supported functions between the corresponding sets. $C^{1,2}$ stands for the class of functions $(t, x) \mapsto \eta_t(x)$, which are $C^1$ %with respect to (
w.r.t. $t$ and $C^2$ w.r.t. $x$. 

$L_{\mathrm p}$, $\rm p \geq 1$, are Lebesgue quotient spaces endowed with the corresponding norms $\|\cdot\|_{L_{\rm p}}$. 

For $A \subseteq \R^n$ and $\rm p\geq 1$, the class $W^{1,2}_{\rm p}(I \times A)$ is introduced 
as a completion of $C^\infty_c(I \times A)$ w.r.t. the norm \(\|\phi\|_{W^{1,2}_{\rm p}} \doteq \|\phi\|_{L_{\rm p}}+ \|\partial_t\phi\|_{L_{\rm p}} + \|\nabla_x\phi\|_{L_{\rm p}}+ \|\nabla^2_{xx}\phi\|_{L_{\rm p}}.\) 

Given a measure space $(\nu, \mathcal X, \mathcal B)$, we denote by $\mathbb E^{\nu} \phi$  the expectation $\langle \nu, \phi \rangle$ of a $\nu$-integrable function $\phi$ w.r.t. $\nu$, and abbreviate $\mathbb E \doteq \mathbb E^{\mathbb P}$. 

$\mathcal P(\mathcal X)$ stands for the collection of probability measures on $\mathcal X$, and $\mathcal P_c(\mathcal X)$ for the set of probability measures with a compact support in $\mathcal X$.

For a Borel measurable function $F\colon \mathcal X \to \mathcal Y$ between metric spaces, the pushforward of a measure $\nu \in \mathcal P(\mathcal X)$ through $F$, $F_{\sharp} \colon \mathcal P(\mathcal X) \to \mathcal P(\mathcal Y)$, is introduced via the action on functions $\phi \in L^1_{F_\sharp \nu}(\mathcal Y)$ as $\mathbb  E^{F_\sharp \nu} \phi \doteq \mathbb E^\nu [\phi\circ F]$.

%\smallskip

We designate a compact and convex set $U \subset \R^m$ to represent the feasible range of control actions. The classes $\bm U_M$ and $\bm U_r$ of admissible Markovian and open-loop (robust) \underline{control strategies} are composed of Borel measurable functions $I \times \R^n \to U$ and $I \to U$, respectively. We endow these sets with the weak* topologies of the corresponding dual Banach spaces $L_\infty(I\times \R^n;\R^m)$ and $L_\infty(I;\R^m)$, provided by the duality $(L_1)^* =L_\infty$. Note that any $u \in \bm U_r$ can be identified with an element $w \in \bm U_M$ such that $w_t(x) \equiv u(t)$. 

% The \underline{class of admissible} \emph{open-loop (robust) control strategies} is introduced as $\bm U_r = L_\infty(I;\bm U_r)$, where $\bm U_r$ stands for the collection of constant maps $\R^n \to U$. Similarly, we define the class $\bm U_\text{M}$ of \emph{Markovian strategies} as $L_\infty(I;\bm U_M)$, where $\bm U_M \doteq L_\infty(\R^n, U)$. Note that \(\bm U_\text{r} \simeq L_\infty(I;U)\) and \(\bm U_\text{M} \subset L_\infty(I \times \R^n; U)\), and  the latter inclusion is strict. We endow the exposed $L_\infty$-spaces with the corresponding weak* topologies of the duality $(L_1)^* = L_\infty$.

Given %a collection $\bm U$ of maps $\R^n \to U$, and 
a function $\mathrm g\colon I \times \R^n \times U  \to \mathcal X$ to some set $\mathcal X$, we define an operator $g\colon I \times \R^n \times \bm U_M  \to \mathcal X$ via the relation $g_t[w](x) \doteq \mathrm g(t, x, w_t(x))$. This operator is instrumental in formulating the data $(R, f, \sigma)$ for the problem $(SP)$, using functions $(\rm R, \rm f, \varsigma)$ that adhere to one of the following sets of (rather excessive)  \underline{assumptions}:
\begin{description}
    \item[$(A_{M})$] ${\rm f} \colon I \times \R^n  \times U \to \R^n$ and ${\rm R} \colon I \times \R^n  \times U \to \R$ are bounded and Borel measurable, and $\varsigma \equiv \sqrt{2\beta} E_n$, where $E_n$ is the unit matrix $(n \times n)$ and $\beta>0$. 
    
    \item[$(A_{r})$] ${\rm f}$ is as above; $\varsigma\colon I \times \R^n  \times U \to \R^{k\times n}$ and ${\rm R}\colon I \times \R^n  \times U \to \R$ %,\\
    are continuous and Lipschitz in the second variable uniformly w.r.t. the other variables. 
%         i.e., there exists a (deterministic) constant
% $\mathrm C>0$ such that 
% \(
% \displaystyle\left|f_t(x, \upsilon) - f_t(y, \upsilon)\right|+\left|\sigma_t(x, \upsilon) - \sigma_t(y, \upsilon)\right| \leq \mathrm C |x-y|
% \) for all $(t, \upsilon) \in I \times U$.
%$\ell\in C^2(\R^n)$. 
%$\sigma$ satisfies the uniform ellipticity assumption \cite{}

%    \item[$(A_2)$] There exists a constant $\xi>0$ such that, for any $e,x \in \R^n$ and $t \in I$, the uniform ellipticity condition holds: \(e^{\rm T} D_t(x) e \geq \xi |e|^2\). 

 %    The function $D$ is continuous. For each $t \in I$ and any compact $K \subset \R^n$, the map $x \mapsto 
 % \sigma_t(x)$ is continuous as $K \to \R^{k\times n}$ uniformly w.r.t. $t \in I$. 
    
 %    The diffusion satisfies a Lipschitz conditions, and the drift, :

%    \item[$(B)$] $U \subset R^m$ is convex and compact.

\end{description}
In addition, we assume that $\mathbb E\big[|X_0|^2\big] < \infty$, while $\ell$ is in $C^2(\R^n)$ and satisfies the quadratic growth condition: 
\(
    \left|\ell(x)\right| \leq {\rm C}\left(1 + |x|^2\right)\) \(\forall (t, x) \in I \times \R^n
\) for some $\rm C>0$.

%\smallskip

We complete this section by presenting some necessary \underline{preliminary facts} from stochastic analysis and stochastic control: Recall that a process $X$ is termed a \emph{strong solution} to the SDE \eqref{SDE}, under some control $\bm u$, if the probability structure $(\Omega, \mathcal F, \mathbb P, W, \mathcal F_t^W)$ is \emph{predefined}, $\mathbb E[(|f[\bm u]|+|\sigma[\bm u]|^2)\circ X] \in L_1(I; \R)$, and the identity \eqref{SDE} holds with $\mathbb P$-probability one (%for $\mathbb P$-almost all (a.a.) $\omega \in \Omega$, 
almost surely, a.s.). A strong solution is a.s. continuous process. It is deemed \emph{strongly unique} if, for any two solutions $X'$ and $X''$, it holds that $\mathbb P(\|X' - X''\|_\infty >0) = 0$. 

Assumptions $(A_r)$ ensure the existence of a strongly unique strong solution for any $\bm u = u \in \bm U_r$ \cite[Thm.~1.3.15]{Pham2021}. This result is foundational in the theory of SDEs and has been extended to the case of Borel measurable drift by \cite[Thm.~1]{Veretennikov1981}. Notably, the conditions $(A_M)$ are sufficient to establish the discussed property for all $\bm u = w \in \bm U_M$. 

In what follows, we shall exclusively deal with strong solutions and omit the epithet ``strong'' for brevity.\footnote{In stochastic analysis, there exists another type of solution to an SDE known as \emph{weak} solution, which incorporates the underlying probability structure as part of the unknown. The existence of a (weakly unique) weak solution is proven under milder assumptions. However, the technical arguments presented in Section~\ref{sec:main} %, which are fundamental to our approach, 
are not applicable to weak solutions.}

 Let $t \in I$, $\bm u \in \bm U_M$, and consider the (linear, unbounded) second-order differential operator $L_t[\bm u]\colon \mathcal D(L) \to C(\R^n)$
\begin{align}
\begin{array}{c}
   % L_t[\upsilon]\colon D(L) \to C(\R^n)  \\[0.2cm]
   L_t[\bm u] \, \phi \doteq \nabla_x \phi \cdot f_t[\bm u] + {\rm Tr}\left(\nabla^2_{xx} \phi \,  D_t[\bm u]\right),   
\end{array}
\label{L}
\end{align} with a dense domain $\mathcal D(L) \doteq C^{2}(\R^n)$, independent of $t$ and $\bm u$. Here, ${\rm Tr}$ denotes the trace of a matrix, and $D \doteq \frac{1}{2} \sigma^{\rm T} \sigma$. Given a solution $X=X[\bm u]$ to \eqref{SDE}, classical \emph{\^{I}to's lemma} says that, for any $\eta \in C^{1,2}(I \times \R^n)$, the composition $Y = \eta\circ X$ is also an \^{I}to process satisfying
\begin{align}
    Y_t = \eta_0(X_0) + & \int_0^t \left\{\partial_s + L_s[\bm u]\right\} \eta(X_s) \d s 
     + M_t, \label{FIto}
\end{align}
where $\displaystyle M_t \doteq  M_t[\bm u] = \int_0^t \nabla_x \eta_s (X_s)^{\rm T} \, \sigma_s[\bm u](X_s) \d W_s$ is a martingale (in particular, $M_0=0$ implies $\mathbb E M =0$ on $I$).

The result remains valid for $\eta \in W^{1,2}_{\rm p}(I \times A)$, where $A \subset \R^n$ is a bounded domain, and ${\rm p}> n+2$ \cite[Thm.~3]{Zvonkin1974}.

Another useful fact is a version \cite[Remark 3.5.5]{Pham2021} of the \emph{Feynman-Kac formula} \cite[Thm.~8.2.1]{Oksendal2010stochastic}: given $u \in \bm U_r$  and  $q \in C(I \times \R^n \times \bm U_r)$, suppose that $p \in C^{1,2}([0,T)\times \R^n)\cap C(I \times \R^n)$ satisfies the quadratic growth condition, 
and solves the inhomogeneous backward Cauchy problem 
\begin{align}
\left\{\partial_s + L_s[\bm u]\right\}p = q_s[\bm u]; \quad p_T=\ell,\label{PDE-back}
\end{align}
where the first relation holds for almost all (a.a.) $s \in I$. Then, $p$ admits the probabilistic representation 
\begin{align}
p_t = \mathbb \ell(X^x_{t, T}) - \int_t^T q_s[\bm u](X^x_{s, T})\d s,\label{FK}
\end{align}
where $t \mapsto X_{s,t}^x$ is a solution of the SDE \eqref{SDE} on the interval $[s, T]$, $s \in [0, T)$, with a deterministic initial condition $X_{s,s}(x)=x \in \R^n$. We recommend \cite[Sec.~1.3.3]{Pham2021} for a brief overview of sufficient conditions for the existence of a $C^{1,2}$ solution to \eqref{PDE-back}, which is applicable to the case $\bm U=\bm U_{r}$.

In \cite[Thm.~1]{Zvonkin1974}, under the assumption that $q \equiv 0$, and $\ell$ is a ``slowly growing'' function belonging to $W^{2}_{\rm p}(A)$ for ${\rm p} > \frac{n+2}{2}$ and any bounded $A \subset \R^n$, it is shown that \eqref{FK} is a unique solution of \eqref{PDE-back}, which is in 
$W^{1,2}_{\mathrm p}(I \times A)$ for each bounded $A \subset \R^n$. Furthermore, the mapping $(t, x) \mapsto \nabla_x p_t(x)$ is uniformly \emph{continuous} (see also \cite[Thm. 3']{Veretennikov1981}). These last two facts are crucial for our analysis in the case $\bm U=\bm U_{M}$.

The final point to emphasize is the transformation of the nonlinear stochastic problem into a state-linear and deterministic one. This reformulation involves an abstract PDE, specifically, the Fokker-Planck-Kolmogorov (FPK) equation in the space of probability measures. Let us discuss this aspect in a separate section.

\subsection{Fokker-Plank Control Framework}

Let $\bm U$ be one of the classes $\bm U_M$ or $\bm U_r$.  The corresponding problem $(SP)$ is equivalent \cite{Anita2021} to a deterministic optimization of the linear form 
\begin{align*}
(DP) \ \ \  \min\left\{\mathbb E^{\mu_T[\bm u]}\ell+ \int_I \mathbb E^{\mu_s[\bm u]} R_s[\bm u] \d s\colon \bm u \in \bm U\right\}
\end{align*}
over %probability 
measure-valued functions
$\mu=\mu[\bm u]\in C(I;\mathcal P(\R^n))\), representing the dynamics \(t\mapsto \mu_t\) of the distribution law 
\begin{equation}
    \mu_t \doteq (X_t)_{\sharp} \mathbb P\label{mu}
\end{equation}
of $X_t$. It is a simple conjecture of \^{I}to's lemma that $\mu$ satisfies the forward Cauchy system for a linear PDE %FPK equation 
\begin{align}
\partial_t \mu = L_t^*[\bm u]\, \mu, \quad \mu_0=\vartheta \doteq (X_0)_\sharp \mathbb P,\label{PDE}
\end{align}
where the operator $L^*$ is a formal adjoint of $L$. 
%The family of operators $(t, \upsilon) \mapsto L_t[\bm u]$ is called the infinitesimal generator of the controlled drift-diffusion process \eqref{SDE}.
%;

Equation \eqref{PDE} is understood in the sense of distributions: for any $\phi \in C^\infty_c(\R^n)$ and $s, t \in I$, it holds
\[
 \mathbb E^{(\mu_t - \mu_\tau)} \phi = \int_\tau^t \mathbb E^{\mu_s} L_s[\bm u] \phi \d s,  
\]
and the initial condition means that $\lim\limits_{t \to 0}\mu_t = \vartheta$ in the corresponding weak* topology. We refer to \cite[\S~9]{bogachev2015fokker} for sufficient conditions of the uniqueness of the measure-valued solution \eqref{mu} to \eqref{PDE}. 

Note that, if $\mu_t$, $t \in I$, are absolutely continuous w.r.t. $\mathcal L^n$, i.e., $\mu_t = \rho_t \mathcal L^n$ with $\rho_t \in L_1(\R^n;\R)$, then  $\rho\colon t \to \rho_t$ is a \emph{weak solution} \cite[Def. 9.2]{bressan2013lecture} of the parabolic PDE
\begin{equation}
    \partial_t \rho = L_t^*[\bm u_t] \rho, \quad \rho_t\big|_{t=0} = \rho_0,
\end{equation}
and this solution is (weakly) unique under relatively weak assumptions \cite{ROCKNER2010435}. 

Some results on the existence of a solution to the optimization problems $(SP)$ and $(DP)$ in the class $\bm U_M$ can be found in \cite[Thm. 6.3]{FlemingRishel} and  \cite[Thm.~4]{Anita2021}.  Under the assumptions $(A_r)$, the existence of a minimizer in the class $\bm U_r$ can be established by the classical continuity-compactness argument, using standard moment estimates \cite[Theorem~1.3.16]{Pham2021} and the Banach-Alaoglu theorem. 

Finally, remark that Markovian strategies, which incorporate a ``feedback'' feature in terms of the SDE \eqref{SDE}, translate to ``open-loop'' controls in the context of the PDE \eqref{PDE}. In the latter setting, the actual feedback mechanism would involve the dependence of $\bm u$ on the measure $\mu$ or its density $\rho$.

\section{$\infty$-ORDER VARIATIONAL ANALYSIS}\label{sec:main}

To be short, we first focus on the Mayer-type functional and uncontrolled diffusion, assuming that $R \equiv 0$ and $\sigma$ is independent of $\bm u$. A brief discussion of the general nonlinear Bolza problem is provided in Section~\ref{sec:C-gen}.

For the rest of the paper, we adhere to a given \emph{reference} control $\bar{\bm u} \in \bm U \in \{\bm U_M, \bm U_r\}$. % subject to an optimality test. 
Our preliminary task is to derive a suitable representation for the increment 
\(\Delta \mathcal I \doteq \mathcal I[\bm u] - \mathcal I[\bar{\bm u}]\) 
of the objective functional concerning another control $\bm u \in \bm U$. In the optimization framework, $\bm u$ takes the role of the unknown \emph{target} strategy to be devised for orchestrating the $\mathcal I$-descent from $\bar{\bm u}$. %, along with  the corresponding necessary optimality condition. 

To streamline the notation, we use a bar to indicate the dependence on $\bar{\bm u}$ and omit mentioning $\bm u$, e.g., $\bar X\doteq X[\bar{\bm u}]$, and $X\doteq X[\bm u]$.

Our reasoning hinges on a simple yet non-standard class of needle-shaped variations: for any $s \in I$, we construct a new control $\bm u \diamond_s \bar{\bm u} \in \bm U$ as 
\begin{equation}
t \mapsto  (\bm u \diamond_s \bar{\bm u})_t \doteq \left\{
\begin{array}{ll}
\bm u_t, & t \in [0,s)\\
\bar{\bm u}_t, & t \in [s,T],
\end{array}
\right.
\label{control-var-gen}    
\end{equation}
and denote
\begin{equation}
\gamma_s \doteq X_{T}[\bm u \diamond_s \bar{\bm u}] = \bar X_{s,T}^{X_{s}}.\label{gamma}
\end{equation}
In view of the moment estimates \cite[Theorem~1.3.16]{Pham2021} and the martingale property, it is evident that the map $s \mapsto \mathbb E\ell(\gamma_s)$ is \emph{Lipschitz continuous} on $I$. Furthermore, by construction,
\begin{equation}
\Delta \mathcal I \doteq \mathcal I[\bm u] - \mathcal I[\bar{\bm u}] = \mathbb E \ell(\gamma_s)\Big|_{s=0}^{s=T} = \int_I \frac{d}{ds}\mathbb E\ell(\gamma_s) \d s.\label{incr-pre}
\end{equation}
Now, let $\bar p$ be a solution to %of the backward Cauchy problem for the dual Kolmogorov equation 
\eqref{PDE-back} with $\bm u =\bar{\bm u}$ and $q \equiv 0$. 
% \begin{align}
% \left\{\partial_s + L[\bar{\bm u}]\right\}p_s = 0, \quad p_T=\ell.\label{PDE-back}
% \end{align}
By the (generalized) Feynman-Kac formula, %$(t, x) \mapsto \bar p_t(x)$ is as
\[
\bar p_t(x) \doteq \mathbb E \ell(\bar X_{t, T}^x) \doteq \mathbb E\left[\ell(\gamma_s)\,|\,X_{0,t}=x\right]. 
\]
Applying the law of total expectation, we can express
\begin{equation}
\mathbb E \bar p_s(X_{s}) \doteq \mathbb E \mathbb E\left[\ell(\gamma_s)\,|\,X_{s}\right] = \mathbb E\ell(\gamma_s).\label{EE}
\end{equation}
Then, utilizing the (generalized) \^{I}to's formula \eqref{FIto}, the stochastic process $s\mapsto \bar p_s(X_{s})$ is demonstrated to satisfy %the SDE
\[
\bar p_t(X_{t}) = \bar p_0(X_{0}) + \int_0^t \left\{\partial_s + L_s[\bm u]\right\}\bar p\circ X_s \d s + M_t,
\]
where $M$ is defined as in Section~\ref{sec-hypo}. By taking $\mathbb E$, applying \eqref{EE},  using Fubini's theorem, and differentiating %\footnote{Recall that the Lipschitz function $s \mapsto \mathbb E\ell(\gamma_s)$ is a.e. differentiable by Rademacher's theorem} 
w.r.t. $s$, we compute, for a.a. $s \in I$:
\begin{align*}
\frac{d}{ds}\mathbb E\ell(\gamma_s)&=\mathbb E\big[\left\{\partial_s + L_s[\bm u]\right\}\bar p\circ X_s\big]\\ 
 & = \mathbb E\big[\left\{L_s[\bm u] - L_s[\bar{\bm u}]\right\}\bar p_s\circ X_s\big]\\
 & = \mathbb E\big[\nabla_x \bar p_s(X_s) \cdot \left(f_s[\bm u] - f_s[\bar{\bm u}]\right)(X_s)\big].
\end{align*}
The difference under the sign of expectation shortly writes 
\[
%\left(\bar H_s[\bm u] - \bar H_s[\bar{\bm u}]\right)\left(X_{s}\right) \doteq  
\bar H_s[\bm u]\left(X_{s}\right) - \bar H_s[\bm u]\left(X_{s}\right),
\]
where
\(
\bar H_s[\bm u](x) \doteq H_s[\bm u](x, \nabla_x \bar p_s(x))
\)
is a contraction to $\psi = \nabla_x  \bar p_s(x)$ of the Hamilton-Pontryagin functional
\[
H_s[\bm u](x, \psi) \doteq \psi \cdot f_s[\bm u](x).
\] 
Substituting the resulting expression into \eqref{incr-pre}, we arrive at the desired representation:
\begin{align}
\Delta \mathcal I = \int_I\mathbb E[\bar H_s[\bm u]\left(X_{s}\right) - \bar H_s[\bar{\bm u}]\left(X_{s}\right)] \d s,\label{increment-X}
\end{align}
which can be reformulated in terms of the laws $\mu_t$ of $X_t$ as
\begin{align}
\Delta \mathcal I & =  \int_I\mathbb E^{\mu_s}\left[\bar H_s[\bm u] - \bar H_s[\bar{\bm u}]\right] \d s.
\label{increment-mu}
\end{align}

Two noteworthy features of the expressions \eqref{increment-X} and \eqref{increment-mu} should be highlighted: i) they are applicable to any pair $(\bar{\bm u}, \bm u)$ of inputs, without any proximity constraints, and ii) they are explicit and \emph{exact}, i.e.,  devoid of any residual terms to be neglected. In parallel with $1^\text{st}$- and $2^\textbf{nd}$-order variations arising from Taylor's expansion of $\mathcal  I$, these formulas can be regarded as $\infty$-order variations of the cost functional at $\bar{\bm u}$, justifying the terminology ``$\infty$-order variational analysis''.  

Similar to the mentioned finite-order variations, our increment formulas provide a characterization of the optimality of the reference control via a particular pointwise optimization problem, detailed below. It should be stressed, however, that utilizing \eqref{increment-X} and \eqref{increment-mu} for optimization purposes implies operating with a specific \emph{feedback mechanism}, as $X$ and $\mu$ correspond to the target control $\bm u$.

\subsection{Markovian Controls}\label{sec:Markov}

We start with the class $\bm U_M$, and adopt hypotheses $(A_M)$. By \(\bar {\rm H}_s\) we denote a contraction of the usual Hamiltonian 
\({\rm H}_s(x, \psi, \upsilon)=\psi \cdot {\rm f}_s \left(x, \upsilon\right)\)
to $\psi = \nabla_x \bar p_s(x)$. In this notation, $\bar H_s[w](x) \doteq  \bar{\rm H}_s(x, w_s(x))$. (In the pseudo-codes below, ${\rm H}^k$ stands  for $\bar {\rm H}$ with $\bar p=p^k$.)

Given \eqref{increment-X}, the following assertion is straightforward.
\begin{proposition}\label{NOC-M}
Assume that $\bar{\bm u} \doteq \bar  w$ is optimal for $(SP)$ within the class $\bm U_\text{M}$. Then, the relation 
\begin{equation}
\mathbb E\bar H_t[w]\left(X_{t}[w]\right) = \mathbb E\bar H_t[\bar w_t]\left(X_{t}[w]\right)\mbox{ for a.a. }t \in I\label{FNOC-w}
\end{equation}
holds for any Borel measurable solution $\bm u \doteq w \in \bm U_\text{M}$ to the mathematical programming problem\footnote{The existence of a Borel measurable function satisfying \eqref{w-min} is a trivial consequence of the continuity of $\bar{\rm H}$, see \cite[Theorems~8.2.11 and 8.1.3]{aubinSetvaluedAnalysis2009}.}
\begin{align}
\bar H_t[w](x) = \min_{\upsilon \in U} \bar {\rm H}_t\left(x, \upsilon\right) \mbox{ for a.a. }t\in I, \ \forall \, x \in \R^n.
\label{w-min}   
\end{align}
\end{proposition}
\begin{proof} Using any control $\bm u = w$ from \eqref{w-min} in \eqref{increment-X}, we have: $\Delta \mathcal I \leq 0$.  The optimality of $\bar w$ thus implies $\Delta \mathcal I=0$, and the assertion follows from non-positivity of the integrand. \end{proof}
% A control $\bar w$ satisfying  Proposition~\ref{NOC-M} is named \emph{extremal}. In this terminology, the value
% \[
% \mathcal E(\bar w) \doteq \mathbb E \int_I \Big(\bar H_s(\bar w_s)(\bar X_s) - \min_{\upsilon \in U} \bar {\rm H}_s\left(\bar X_s, \upsilon\right)\Big) \d s \geq 0
% \] 
% serves a measure of non-extremality of $\bar w$.

This result marks a conceptual difference with the PMP: Proposition~\ref{NOC-M} involves not only the tested control $\bar w$ but also an additional \emph{comparison} control $w$, derived from $\bar w$. This feature complicates the use of Proposition~\ref{NOC-M} per se. However, construction \eqref{w-min}
gives rise to Algorithm~\ref{alg1}, which is \emph{free of intrinsic parameters}, while generating a sequence $\{w^k\}$ with a monotonicity property:  $\mathcal I[w^{k+1}] \leq \mathcal I[w^k] \doteq \mathcal I^k$.  
\begin{algorithm}
\caption{Optimal Markovian control}
\label{alg1}
\KwData{$\bar w \in \bm U_M$ (initial guess), $\varepsilon>0$ (tolerance)}
\KwResult{$\{w^k\}_{k \geq 0} \subset \bm U_M$ such that $I[w^{k+1}] < I[w^{k}]$}
$k \gets 0$;
$w^0 \gets \bar w$\;
\Repeat{$I[w^{k-1}] - I[w^{k}] < \varepsilon$}{
$p^{k} \gets p[w^k]$; $w\in \arg\min {\rm H}^k_s$\;
$w^{k+1} \gets w$; 
$k \gets k+1$\;
  }
\end{algorithm}

Denote by $\mathcal E[\bar w, w]$ the negative right-hand side of \eqref{increment-X}. The map $(\bar w, w) \mapsto \mathcal E[\bar w, w]$ defines a non-negative functional $\mathcal E\colon \bm U_M \times \bm U_M \to \R$, which measures the violation of the condition \eqref{FNOC-w}. Observing that $\{\mathcal I^k\}$ is bounded, and therefore, $\mathcal E[w^{2k}, w^{2k+1}] \doteq \mathcal I^{2k}- \mathcal I^{2k+1} \to 0$,  
%_{k \to \infty}
we can apply the Tychonoff and   Banach-Alaoglu theorems to conclude that $\{(w^{2k}, w^{2k+1})\}$ converges, up to a subsequence, to a pair satisfying \eqref{FNOC-w}.

\subsection{Open-Loop Controls}\label{sec:robust}

Now, we turn to the class $\bm U_r$ of robust control strategies. Though this choice implies significantly reduced controllability compared to $\bm U_M$, its practical implementation is much simpler. Moreover, as discussed in \cite{Brockett2007,Breitenbach2020}, open-loop controls can serve as reasonable approximations of Markovian strategies. A numerical verification of this thesis is conducted by \cite{Breitenbach2020} using an affine-bilinear approximation.

A narrative akin to the previous paragraph can be extended to the law dynamics $t \mapsto \mu_t$. Assuming $(A_r)$, the formula \eqref{increment-mu} becomes instrumental in assessing the optimality of $\bar{\bm u} = \bar{u} \in \bm U_\text{r}$ proved similar to Proposition~\ref{NOC-M}.
\begin{proposition}
Assume that the pair $(u, \mu)$, $u \in \bm U_r$, $\mu=\mu[u]$, satisfies the relation
\begin{align}
 \mathbb E^{\mu_t}\,\bar {\rm H}_t\left(u_{t}\right) =\min_{\upsilon \in U} \mathbb E^{\mu_t}\,\bar {\rm H}_t\left(\upsilon\right)\mbox{ for a.a. }t \in I.\label{u-min} 
\end{align}
If $\bar u$ is optimal, then \eqref{u-min} also holds for $u=\bar u$.
\end{proposition}

Note that \eqref{u-min} is, in fact, a type of operator equation on $\bm U_r$. In this concise paper, we shall refrain from rigorously addressing the existence of its solution, which is not straightforward and can potentially be proved, for instance, by applying Kakutani-Fan's fixed-point theorem. Instead, let us explore a constructive approach to (approximately) solving \eqref{u-min} via the feedback principle: let $\nu \in \mathcal{P}(\mathbb{R}^n)$, and $\bar v_t[\nu]$ be a solution to the %mathematical programming 
problem
\begin{align}\label{fb}
\mathbb E^{\nu}\bar {\rm H}_t\left(\bar v_t[\nu]\right) =\min_{\upsilon \in U} \mathbb E^{\nu}\bar {\rm H}_t\left(\upsilon\right)\mbox{ for a.a. }t \in I.
\end{align}
The mapping $\bar v\colon (t, \nu) \mapsto \bar v_t[\nu]$ acts as a \emph{feedback control} of \eqref{PDE}, or a \emph{law-feedback control} of \eqref{SDE}. Substituting $\bar v$ into \eqref{PDE} results in a \emph{nonlocal} FPK equation. Assuming that the latter possesses a unique solution $\hat{\mu}=\hat\mu[\bar v]$, and setting $u(t) \doteq \bar v_t[\hat{\mu}_t]$, we obtain a new control $u \in \bm U_r$ with the property \eqref{u-min} followed by the desired inequality $\Delta \mathcal{I} \leq 0$. Note that, in general, $\bar v[\cdot]$ is discontinuous. A solution to the corresponding backfed PDE can be designed by the Krasovskii-Subboting sampling method in the spirit of \cite{SChPP-2022}.  

An iterative implementation of this idea is outlined in Algorithm~\ref{alg2}. The convergence analysis can be elaborated similarly to \cite[Appendix~B]{SPP-2023}.
\begin{algorithm}
\caption{Optimal robust control}
\label{alg2}
\KwData{$\bar u \in \bm U_r$ (initial guess), $\varepsilon>0$ (tolerance)}
\KwResult{$\{u^k\}_{k \geq 0} \subset \bm U_r$ such that $I[u^{k+1}] < I[u^{k}]$}
$k \gets 0$;
$u^0 \gets \bar u$\;
%$N \gets n$\;
\Repeat{$I[u^{k-1}] - I[u^{k}] < \varepsilon$}{
%   \eIf{$N$ is even}{
%     $X \gets X \times X$\;
%     $N \gets \frac{N}{2}$ \Comment*[r]{This is a comment}
%   }{\If{$N$ is odd}{
%       $y \gets y \times X$\;
%       $N \gets N - 1$\;
%     }
%$\p_k \gets p[u_k]$\;
$p^{k} \gets p[u^k]$; $v^k_s[\nu]\in \arg\min \mathbb E^{\nu}{\rm H}^k_s$\; $\mu^{k+1} \gets \hat \mu[v^k]$;
$u^{k+1} \gets v^k[\mu^{k+1}]$;
$k \gets k+1$\;
  }
\end{algorithm}

In conclusion, we emphasize again that our construction of law-feedback control relies solely on statistics derived from preliminary experiments, accessible through the model \eqref{SDE} (provided, in particular, by a solution to the PDE \eqref{PDE-back}). It does not utilize any feedback information pertaining to a particular sample path, chosen by the stochastic process in the course of practical implementation.

\subsection{Local Optimality. Geometric Intuition. Relation to PMP}\label{sec:geometric}

For fixed $\omega \in \Omega$, the parametrization 
\(
s \mapsto \gamma_s(\omega)
\), defined in \eqref{gamma},
represents, a.s., a curve in the set 
\[
\mathcal R_T(\omega)\doteq \{X_T[\bm u](\omega)\colon \bm u \in \bm U\},
\] of points, reachable by all $\omega$-paths of the controlled SDE \eqref{SDE} (``sample-reachable'' set). Certainly, since $\bm u \diamond_s \bar{\bm u} \in \bm U$, it is evident that, a.s., $\gamma_s(\omega) \in \mathcal R_T(\omega)$ for any $s \in I$, while $\gamma_0(\omega) = \bar X_T(\omega)$ and $\gamma_T(\omega) = X_T(\omega)$ by construction. 

Building on this observation, we introduce an original concept of local minimum for problem $(SP)$ that applies to both types $\bm U_r$ and $\bm U_M$ of control inputs. This concept does not rely on any norm on the spaces of state trajectories and/or control functions, and consequently, proves to be stronger than the usual ``strong'' or Pontryagin's minimum associated with the standard class of needle-shaped variations.
\begin{definition}
A process 
$\gamma\colon %(s, \omega) \mapsto \gamma_s(\omega)$,$
I \times \Omega \to \R^n$, is said to be a \emph{curve of expected monotone decrease from $x \in \R^n$ w.r.t. $\ell$} if the following holds: i) $\gamma$ is $\mathcal F_T^{W,X_0}$-measurable in $\omega$ for all $s \in I$, ii) $\gamma$ is a.s. continuous in $s$, iii) $\gamma_0=x$ a.s., and iv) $s \mapsto \mathbb E\ell(\gamma_s)$ is strictly decreasing on $I$. 

We call a control $\bar u \in \bm U$ \emph{locally optimal} for $(SP)$ if there are \emph{no} curves $\gamma$ of the expected monotone decrease from 
$\bar X_T$ such that $\gamma_s \in \mathcal R_T$ a.s., for all $s \in I$.       
\end{definition}

Essentially, our approach boils down to suggesting a somewhat simplest class of a.s. Lipschitz curves on the sample-reachable set that perform a \emph{guaranteed} ``expected non-ascent'' from $\bar X_T$. In the non-stochastic framework, there are toy examples \cite{SChPP-2022} showing that necessary conditions, relying on this idea, can discard non-optimal PMP extrema. To construct such examples for the stochastic problem would be an interesting challenge. 

\subsection{Bolza Problem. Nonlinear Functionals}\label{sec:C-gen}

Generalizing the presented results to the Bolza problem is straightforward: by setting $\bar p$ in the form \eqref{FK} with $\bm u=\bar{\bm u}$ and $q=R$, one replicates the logic of Section~\ref{sec:main} with the corresponding re-definition of the maps $\bar{\rm H}$ and $\bar H$.

In fact, it is possible to extend \eqref{increment-mu} to a wide class of \emph{nonlinear} cost functionals, $\bm \ell, \bm R_s \colon \mathcal P_c(\R^n) \to \R$, possessing a differentiable \emph{intrinsic derivative} (the derivative along vector fields) \cite[Definition~2.2.2]{CardMaster2019}. For the case of random ODEs, such an extension is performed in \cite{CPSA2023}. 

% In applications with the affine dependence on $\bm u$, the explicit pointwise bounds on the control are often disregarded ($U$ is assumed to be a ball of sufficiently big radius). Instead, $\bm u$ is penalized by the ``energetic'' running cost $R$, defined as in Sec.~\ref{sec-hypo} via an auxiliary function $\mathrm R(t, x, \upsilon) \equiv \frac{\alpha}{2}|\upsilon|^2$, or directly as $R[\bm u] = \frac{\alpha}{2}|\bm u|^2$, $\alpha >0$. In the latter case, comparison Markovian controls arise from the pointwise minimization of the function $\frac{\alpha}{2}|w|^2 + \mathbb E^{\mu}\bar H_s[w]$, and become law-dependent. 

In practical scenarios featuring affine dependence on $\bm u$, geometric control bounds, $\bm u_t \in U$ a.e. $t \in I$, are often omitted ($U$ is assumed to be a large-radius ball). Instead, $\bm u$ is penalized by an ``energetic'' running cost functional. For $\bm u =w$ this penalty is defined either through an auxiliary function $\mathrm R \equiv \frac{\alpha}{2}|\upsilon|^2$ as in Sec.~\ref{sec-hypo}, or directly as $\frac{\alpha}{2}\int_I \|w_t\|^2_{L_2} \d t$, $\alpha >0$. In the latter case, if the initial distribution $\vartheta$ is absolutely continuous, a comparison control $w$ is defined via the \emph{unique} minimizer $w_t[\rho](x)$ of $\upsilon \mapsto \frac{\alpha}{2}|\upsilon|^2 + \bar{\mathrm{H}}_t(x, \upsilon) \rho(x)$, which depends on the density $\rho$ of the probability law.

%Its adaptation to the problem $(DP)$ could employ the chain rule  \cite[Thm.~3.5]{cardaliaguetAnalysisSpaceMeasures2019}. 

% A prominent example is given by the functional
% \(
% \bm \ell[\nu] \doteq \frac{1}{2} \mathbb E^{\nu}|x- \xi|^2+\frac{1}{2}\sum_{i,j=1}^n \left(\mathbb E^{\nu}[x_i x_j] - \mathbb E^{\nu}[x_i] \mathbb E^{\nu}[x_j] - \zeta_{ij}\right)^2,
% \)
% which captures the ``mean-covariance steering'' problem to some target data $\xi \in \R^n$, $\zeta=(\zeta_{ij}) \in \R^{n \times n}$ (cf. \cite{Hotz}). 

% \subsection{Diffusion on Manifolds}

% Nick: \textbf{Since computational parts shall probably deal with a periodic model, can we write a phrase about translation of this cuisine to $\mathbb T^n$}.

\section{NUMERIC EXPERIMENT}\label{sec:numeri}

To exemplify our approach, we tackle a stochastic version of an optimal control problem \cite[Sec.~III.A]{SPP-2023} for the Theta model, a simple model capturing bursting behavior of excitable neurons. The model features a phase variable \(x\in \mathbb S^1 \) and an excitability parameter \( \eta\in \mathbb{R}\), governed by the SDE \eqref{SDE}, with
\(
f[\bm u](x)=(1-\cos x) +(1+\cos x)\left(\eta+\bm u\right)\), and \(\sigma \equiv \sqrt{2\beta}\). To highlight novel aspects in the stochastic setting, we formulate the problem using Markovian strategies $\bm u = w$, representing external excitations.  Our objective is to steer the phase of a neuron from random initial data, characterized by a joint probability distribution in $(x, \eta)$, to a specified value $\check x \in \mathbb S^1$ at a given moment $T$,  while optimizing resource utilization: $\ell(x)=  1 - \cos(x-\check x)$, and the running cost functional is $\frac{\alpha}{2}\int_I\|w_t\|^2_{L_2} \d t$. The problem is cast into the form $(DP)$ and solved by Algorithm~\ref{alg1}. %, involving law-feedback controls. 

For our experimental setup, we choose $T=6$,  $\alpha=1$, $\beta = 0.5$, $\check x =\pi$, and $u^0 \equiv 0$. The PDEs \eqref{PDE} and \eqref{PDE-back} are solved by the standard pseudospectral method, and time integration is performed using the 4$^\text{th}$-order Runge-Kutta scheme. Note that the backfed equation \eqref{PDE} turns out to be nonlocal. 

The probability density function (PDF) of the initial distribution $\vartheta$ is depicted in Fig.~\ref{fig:mu} (upper panel). The minimization history spans 3 iterations: $\mathcal I \approx 11.46\to 2.88 \to 2.4 \to 2.31$. Notably, rapid convergence towards a meaningful solution with a reasonable control cost is evident. Fig.~\ref{fig:mu} (bottom panel) illustrates the resulting terminal PDF. %Fig.~\ref{fig:u} depicts the evolution of the resource consumed by the controller.

\begin{figure}[!ht]
  \centering
  \vspace{0.2cm}\includegraphics[width=0.43\textwidth]{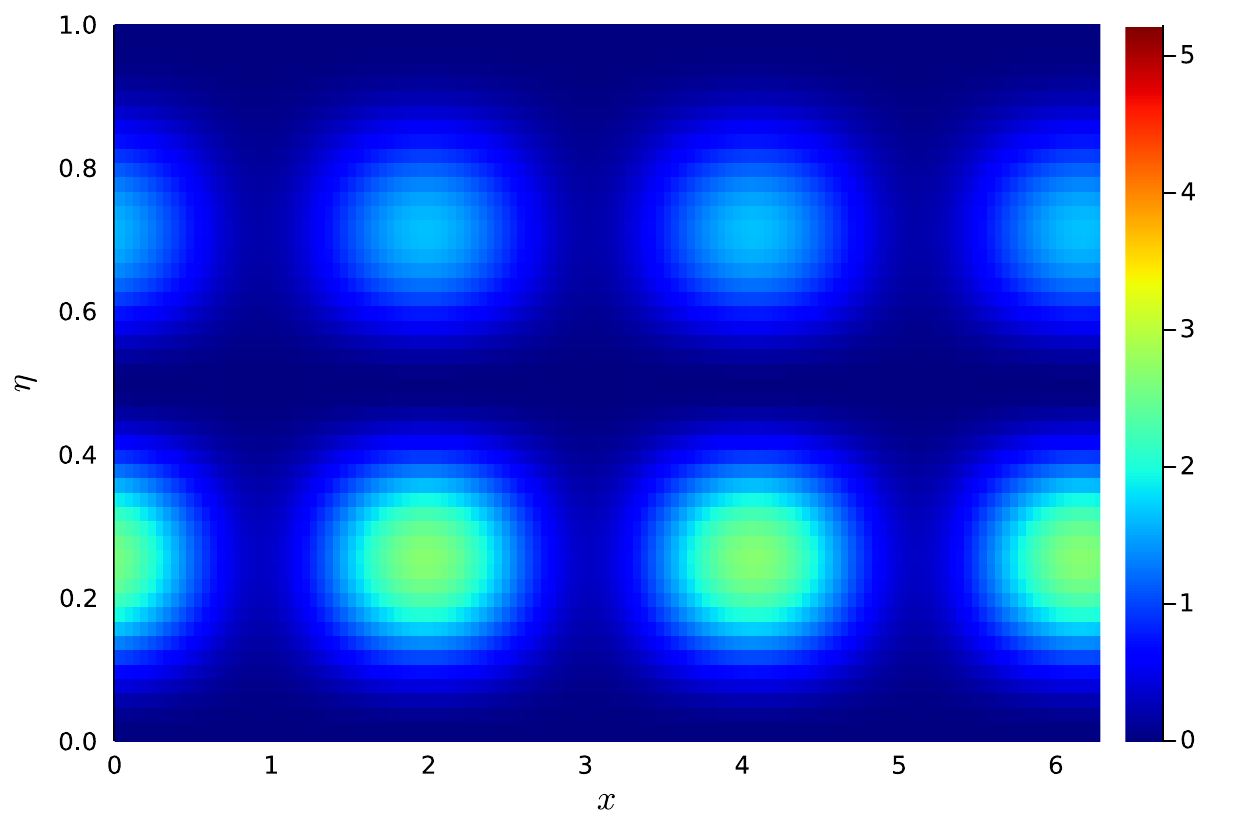}\vspace{-0.34cm}
  \includegraphics[width=0.43\textwidth]{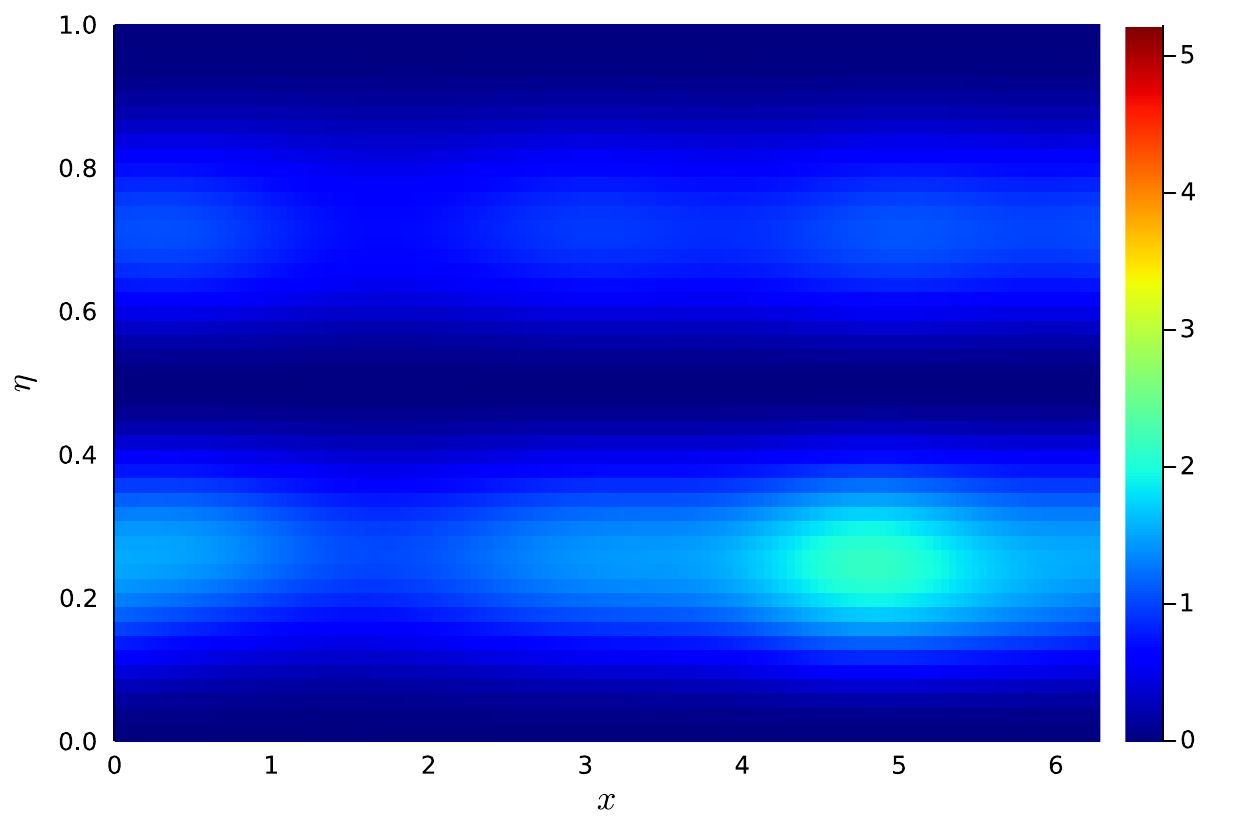}\vspace{-0.34cm}
  \includegraphics[width=0.43\textwidth]{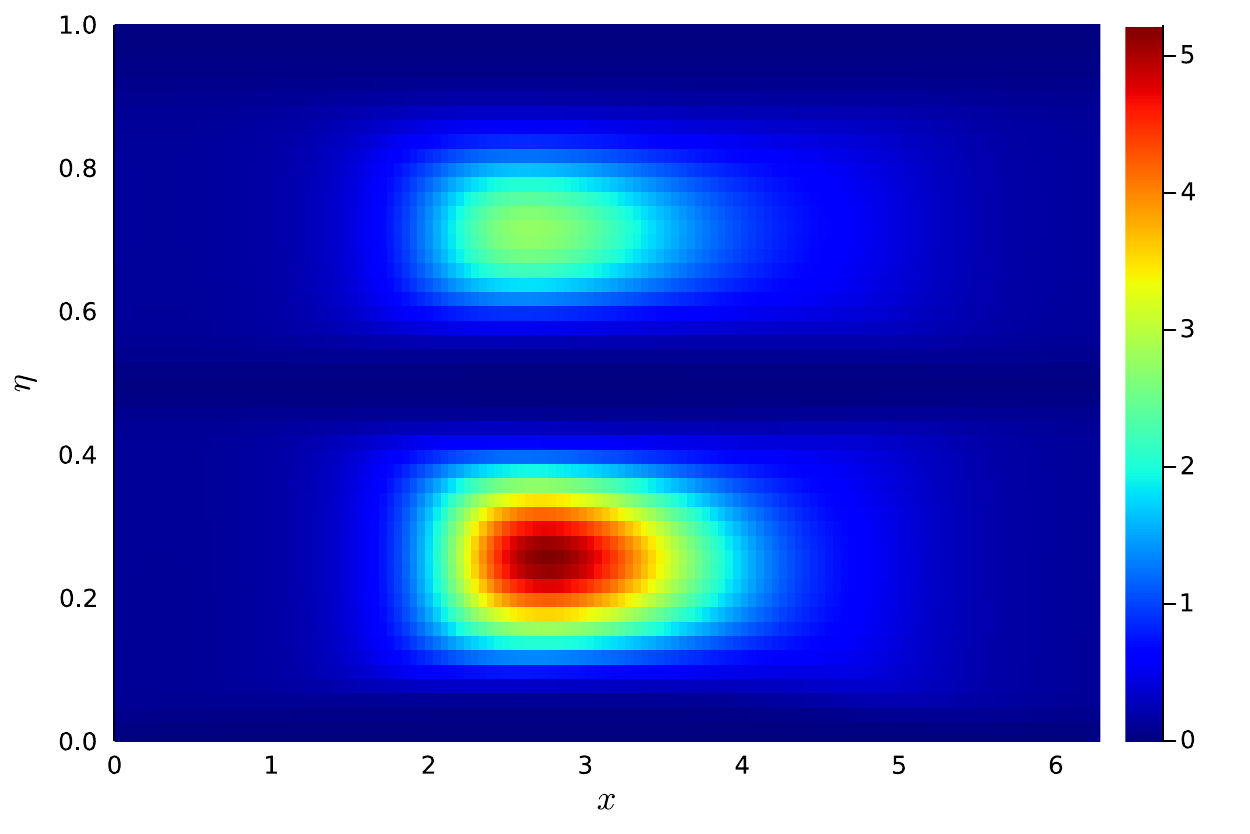}\vspace{-0.34cm}
  \caption{Theta model: snapshots of the ``optimal'' PDF at time moments $t=0; 0.5; 6$ (from top to bottom).}
  \label{fig:mu}
\end{figure}

% \begin{figure}[!ht]
%   \centering
%   \includegraphics[width=0.28\textwidth]{unorm.pdf}
%   \caption{Theta model: time-evolution of the ``optimal'' value $\| u \|_{L_2}^2$.}
%   \label{fig:u}
% \end{figure}

%\subsection{Few-Shot Learning of Neural SDEs. System of uncoupled FPK equations}

% \section{CONCLUSION}

% Will depend on the remaining space 

% \subsection{Nonlocal FPK-Equation}
% \label{sec:C-gen}

% \subsection{Other Types of Control Variations. Constrained problems}

% \appendix

% %\section{Appendix}

% \subsection{Derivation of the Increment Formula}\label{sec:appendix-incr}

\bibliographystyle{IEEEtran}%{IEEEconf}
\bibliography{IEEEabrv,Staritsyn-full}

\def\cprime{$'$}
\begin{thebibliography}{10}
\providecommand{\url}[1]{#1}
\csname url@samestyle\endcsname
\providecommand{\newblock}{\relax}
\providecommand{\bibinfo}[2]{#2}
\providecommand{\BIBentrySTDinterwordspacing}{\spaceskip=0pt\relax}
\providecommand{\BIBentryALTinterwordstretchfactor}{4}
\providecommand{\BIBentryALTinterwordspacing}{\spaceskip=\fontdimen2\font plus
\BIBentryALTinterwordstretchfactor\fontdimen3\font minus
  \fontdimen4\font\relax}
\providecommand{\BIBforeignlanguage}[2]{{%
\expandafter\ifx\csname l@#1\endcsname\relax
\typeout{** WARNING: IEEEtran.bst: No hyphenation pattern has been}%
\typeout{** loaded for the language `#1'. Using the pattern for}%
\typeout{** the default language instead.}%
\else
\language=\csname l@#1\endcsname
\fi
#2}}
\providecommand{\BIBdecl}{\relax}
\BIBdecl

\bibitem{BENSOUSSAN1983387}
A.~Bensoussan, ``Stochastic maximum principle for distributed parameter
  systems,'' \emph{Journal of the Franklin Institute}, vol. 315, no.~5, pp.
  387--406, 1983.

\bibitem{Peng}
S.~Peng, ``A general stochastic maximum principle for optimal control
  problems,'' \emph{SIAM Journal on Control and Optimization}, vol.~28, no.~4,
  pp. 966--979, 1990.

\bibitem{Hu2021}
Y.~Hu, \emph{Stochastic Maximum Principle}.\hskip 1em plus 0.5em minus
  0.4em\relax Springer International Publishing, 2021, pp. 2186--2190.

\bibitem{Pham2021}
H.~Pham, \emph{{Continuous-time Stochastic Control and Optimization with
  Financial Applications}}, 2021.

\bibitem{WeinanJiequn}
W.~E, J.~Han, and Q.~Li, ``A mean-field optimal control formulation of deep
  learning,'' \emph{Research in the Mathematical Sciences}, vol.~6, no.~1,
  p.~10, 2018.

\bibitem{Siska}
J.-F. Jabir, D.~Siska, and L.~Szpruch, ``Mean-field neural odes via relaxed
  optimal control,'' 2019.

\bibitem{Annunziato2013}
M.~Annunziato and A.~Borz{\`{i}}, ``{A Fokker – Planck control framework for
  multidimensional},'' \emph{Journal of Computational and Applied Mathematics},
  vol. 237, no.~1, pp. 487--507, 2013.

\bibitem{Roy2016}
S.~Roy, M.~Annunziato, and A.~Borz{\`{i}}, ``{A Fokker – Planck Feedback
  Control-Constrained Approach for Modeling Crowd Motion A Fokker – Planck
  Feedback Control-Constrained Approach},'' \emph{Journal of Computational and
  Theoretical Transport}, vol. 4309, no. July, 2016.

\bibitem{Roy2018}
S.~Roy, M.~Annunziato, A.~Borz{\`{i}}, and C.~Klingenberg, ``{A Fokker –
  Planck approach to control collective motion},'' \emph{Computational
  Optimization and Applications}, vol.~69, no.~2, pp. 423--459, 2018.

\bibitem{Breitenbach2020}
T.~Breitenbach and A.~Borz{\`{i}}, \emph{{The Pontryagin maximum principle for
  solving Fokker – Planck optimal control problems}}.\hskip 1em plus 0.5em
  minus 0.4em\relax Springer US, 2020, vol.~76, no.~2.

\bibitem{Fleig2016}
A.~Fleig and R.~Guglielmi, ``{Bilinear Optimal Control of the Fokker-Planck
  Equation},'' \emph{IFAC-PapersOnLine}, vol.~49, no.~8, pp. 254--259, 2016.

\bibitem{Fleig2017}
------, ``{Optimal Control of the Fokker-Planck Equation with Space-Dependent
  Controls},'' \emph{Journal of Optimization Theory and Applications}, pp.
  1--20, 2017.

\bibitem{Borzi2011}
A.~Borz{\`{i}} and V.~Schulz, \emph{{Computational Optimization of Systems
  Governed by Partial Differential Equations}}, 2011.

\bibitem{Anita2021}
c.~L. Aniţa, ``{Optimal Control of Stochastic Differential Equations via
  Fokker–Planck Equations},'' \emph{Applied Mathematics and Optimization},
  vol.~84, pp. 1555--1583, 2021.

\bibitem{krylov2008controlled}
N.~Krylov and A.~Aries, \emph{Controlled Diffusion Processes}, ser. Stochastic
  Modelling and Applied Probability.\hskip 1em plus 0.5em minus 0.4em\relax
  Springer Berlin Heidelberg, 2008.

\bibitem{Annunziato2014}
M.~Annunziato, A.~Borz{\`{i}}, F.~Nobile, and R.~Tempone, ``{On the Connection
  between the Hamilton-Jacobi-Bellman and the Fokker-Planck Control
  Frameworks},'' \emph{Applied Mathematics}, no. September, pp. 2476--2484,
  2014.

\bibitem{yong1999stochastic}
J.~Yong and X.~Zhou, \emph{Stochastic Controls: Hamiltonian Systems and HJB
  Equations}, ser. Stochastic Modelling and Applied Probability.\hskip 1em plus
  0.5em minus 0.4em\relax Springer New York, 1999.

\bibitem{Sinigaglia2021OptimalCO}
C.~Sinigaglia, F.~Braghin, and S.~Berman, ``Optimal control of velocity and
  nonlocal interactions in the mean-field {Kuramoto} model,'' in \emph{2022
  American Control Conference (ACC)}, 2022, pp. 290--295.

\bibitem{ma1999forward}
J.~Ma and J.~Yong, \emph{Forward-Backward Stochastic Differential Equations and
  Their Applications}, ser. Forward-backward Stochastic Differential Equations
  and Their Applications.\hskip 1em plus 0.5em minus 0.4em\relax Springer,
  1999, no. № 1702.

\bibitem{AMO}
R.~Chertovskih, N.~Pogodaev, and M.~Staritsyn, ``Optimal control of nonlocal
  continuity equations: Numerical solution,'' \emph{Applied Mathematics {\&}
  Optimization}, vol.~88, no.~3, p.~86, 2023.

\bibitem{MarigQuin2018}
A.~Marigonda and M.~Quincampoix, ``Mayer control problem with probabilistic
  uncertainty on initial positions,'' \emph{Journal of Differential Equations},
  vol. 264, no.~5, pp. 3212--3252, 2018.

\bibitem{SChPP-2022}
M.~Staritsyn, N.~Pogodaev, R.~Chertovskih, and F.~L. Pereira, ``Feedback
  maximum principle for ensemble control of local continuity equations: An
  application to supervised machine learning,'' \emph{IEEE Control Systems
  Letters}, vol.~6, pp. 1046--1051, 2022.

\bibitem{SPP-2023}
M.~Staritsyn, N.~Pogodaev, and F.~L. Pereira, ``Linear-quadratic problems of
  optimal control in the space of probabilities,'' \emph{IEEE Control Systems
  Letters}, vol.~6, pp. 3271--3276, 2022.

\bibitem{CPSA2023}
R.~Chertovskih, N.~Pogodaev, M.~Staritsyn, and A.~P. Aguiar, ``Optimal control
  of distributed ensembles with application to {Bloch} equations,'' \emph{IEEE
  Control Systems Letters}, vol.~7, pp. 2059--2064, 2023.

\bibitem{Veretennikov1981}
A.~J. Veretennikov, ``On strong solutions and explicit formulas for solutions
  of stochastic integral equations,'' \emph{Mathematics of the USSR-Sbornik},
  vol.~39, no.~3, p. 387, 1981.

\bibitem{Zvonkin1974}
A.~K. Zvonkin, ``A transformation of the phase space of a diffusion process
  that removes the drift,'' \emph{Mathematics of the USSR-Sbornik}, vol.~22,
  no.~1, p. 129, 1974.

\bibitem{Oksendal2010stochastic}
B.~{\O}ksendal, \emph{Stochastic Differential Equations: An Introduction with
  Applications}, ser. Universitext.\hskip 1em plus 0.5em minus 0.4em\relax
  Springer Berlin Heidelberg, 2010.

\bibitem{bogachev2015fokker}
V.~Bogachev, N.~Krylov, M.~R{\"o}ckner, and S.~Shaposhnikov,
  \emph{Fokker-Planck-Kolmogorov Equations}, ser. Mathematical Surveys and
  Monographs.\hskip 1em plus 0.5em minus 0.4em\relax American Mathematical
  Society, 2015.

\bibitem{bressan2013lecture}
A.~Bressan, \emph{Lecture Notes on Functional Analysis: With Applications to
  Linear Partial Differential Equations}, ser. Graduate studies in
  mathematics.\hskip 1em plus 0.5em minus 0.4em\relax American Mathematical
  Society, 2013.

\bibitem{ROCKNER2010435}
M.~Röckner and X.~Zhang, ``Weak uniqueness of fokker–planck equations with
  degenerate and bounded coefficients,'' \emph{Comptes Rendus Mathematique},
  vol. 348, no.~7, pp. 435--438, 2010.

\bibitem{FlemingRishel}
W.~H. Fleming and R.~W. Rishel, \emph{Deterministic and stochastic optimal
  control}, ser. Appl. Math. (N. Y.).\hskip 1em plus 0.5em minus 0.4em\relax
  Springer, New York, 1975, vol.~1.

\bibitem{aubinSetvaluedAnalysis2009}
J.~P. Aubin and H.~Frankowska, \emph{Set-valued analysis}, ser. Modern
  {Birkhauser} classics.\hskip 1em plus 0.5em minus 0.4em\relax Birkhäuser,
  2009.

\bibitem{Brockett2007}
R.~Brockett, ``Optimal control of the liouville equation,'' \emph{Studies in
  Advanced Mathematics}, vol.~39, pp. 23--35, 2007.

\bibitem{CardMaster2019}
P.~Cardaliaguet, F.~Delarue, J.-M. Lasry, and P.-L. Lions, \emph{The Master
  Equation and the Convergence Problem in Mean Field Games}, ser. Ann.
  {{Math}}. {{Stud}}.\hskip 1em plus 0.5em minus 0.4em\relax {Princeton
  University Press}, 2019, vol. 201.

\end{thebibliography}

%\section*{\#STOPWARUKRAINE}

\end{document}